\documentclass{amsart}
\newtheorem{theorem}{Theorem}[section]

\usepackage{latexsym}
\usepackage{amssymb,amsmath,amsfonts}
\usepackage[colorlinks=true]{hyperref}
\usepackage{graphicx}
\usepackage{array}

\newtheorem{ass}[theorem]{Assumption}
\newtheorem{alg}[theorem]{Algorithm}

\newtheorem{lemma}[theorem]{Lemma}
\newtheorem{cor}[theorem]{Corollary}

\theoremstyle{definition}

\theoremstyle{remark}
\newtheorem{remark}[theorem]{Remark}

\numberwithin{equation}{section}
 \global\parskip 6pt \oddsidemargin
0.1cm \evensidemargin 0.4cm \headheight 0.1cm \topmargin -2.0cm
\begin{document}
\title
[Monotone Inclusion Problems in Reflexive Banach Spaces]
 {Convergence of Two Simple Methods for Solving Monotone Inclusion Problems in Reflexive Banach Spaces}

\author{Chinedu Izuchukwu$^1$, Simeon Reich$^{2}$, Yekini Shehu$^3$}

\keywords{Bregman resolvent; forward-backward splitting method; gas market; generalized Nash equilibrium problem; maximal monotone operator; monotone inclusion; reflexive Banach space. \\
{\rm 2010} {\it Mathematics Subject Classification}: 47H09; 47H10; 49J20; 49J40\\\\
{$^{1, 2}$Department of Mathematics, The Technion -- Israel Institute of Technology, 32000 Haifa, Israel.}\\
{$^{3}$College of Mathematics and Computer Science, Zhejiang Normal University, Jinhua 321004, People’s Republic of China.}\\
{$^{1}\mbox{izuchukwu}\_c$@yahoo.com; chi.izuchukw@campus.technion.ac.il}\\
{$^2$sreich@technion.ac.il}\\
{$^{3}$yekini.shehu@zjnu.edu.cn}}
\begin{abstract}
 We propose two very simple methods, the first one with constant step sizes and the second one with self-adaptive step sizes, for finding a zero of the sum of two monotone operators in real reflexive Banach spaces. Our methods require only one evaluation of the single-valued operator at each iteration. Weak convergence results are obtained when the set-valued operator is maximal monotone and the single-valued operator is Lipschitz continuous, and strong convergence results are obtained when either one of these two operators is required, in addition, to be strongly monotone. We also obtain the rate of convergence of our proposed methods in real reflexive Banach spaces. Finally, we apply our results to solving generalized Nash equilibrium problems for gas markets.\\


\end{abstract}

\maketitle

\section{Introduction}\label{Se1}
\noindent Let $\mathbb{X}^*$  be the dual of a real reflexive Banach space $\mathbb{X}$ and let $\mathcal{B}:\mathbb{X}\to 2^{\mathbb{X}^*}$ be a set-valued operator. The Null Point Problem (NPP) is formulated as follows:
\begin{eqnarray}\label{NPP}
\mbox{Find} ~v\in \mathbb{X} ~\mbox{such that}~ 0^* \in  \mathcal{B}v.
\end{eqnarray}
\noindent This problem is very important in optimization theory and related fields. Many optimization problems such as minimization problems, equilibrium problems, and saddle point problems, can be modelled as Problem \eqref{NPP}. For instance, if $\mathcal{B}$ is the subdifferential of a proper, convex and lower semicontinuous function, then Problem \eqref{NPP} is equivalent to the problem of minimizing this convex function. Also, Problem \eqref{NPP} describes the equilibrium or stable state of an evolution system governed by the  operator $\mathcal{B}$, which is very important in ecology, physics and economics, as well as in other fields (see \cite{C}). One of the most popular methods for solving Problem \eqref{NPP} is the Proximal Point Algorithm (PPA), which was introduced in Hilbert spaces by Martinet \cite{27ZMO} in 1970. The PPA was further developed by Rockafellar \cite{31ZMO} and Bruck and Reich \cite{BR} in 1976 and 1977, respectively. Since then, the PPA has been modified by several authors in order to solve Problem \eqref{NPP} (see, for example,  \cite{BMR, NR, RR, RS1, RS2}). We mention, in particular, the paper by Reich and Sabach \cite{RS1}. The authors of this paper introduced the following modification of the PPA for solving Problem \eqref{NPP} in a reflexive Banach space:
\begin{eqnarray}\label{RS}
\begin{cases}
x_1\in \mathbb{X}\\
y_n=\mbox{Res}^g_{\mu_n \mathcal{B}}(x_n)\\
C_{n}=\{z\in \mathbb{X}: D_g(z, y_n)\leq D_g(z, x_n)\}\\
Q_n=\{z\in \mathbb{X}: \langle \nabla g(x_1)-\nabla g(x_n), z-x_n\rangle\leq 0\}\\
x_{n+1}=\mbox{Proj}^g_{C_{n}\cap Q_n}(x_1), ~n\geq 1,
\end{cases}
\end{eqnarray}
where $\{\mu_n\}$ is a given sequence of positive real numbers, $\mbox{Res}^g_{\mathcal{B}}$ is the resolvent of $\mathcal{B}$, $\nabla g$ is the gradient of $g$ and $\mbox{Proj}^g_{\mathcal{C}}$ is the Bregman projection of $\mathbb{X}$ onto a nonempty, closed and convex subset $\mathcal{C}$ of $\mathbb{X}$.\\
  In the same paper, the authors studied two extensions of Algorithm \eqref{RS} that solve Problem \eqref{NPP} for finitely many  maximal monotone operators in real reflexive Banach spaces.

\noindent Continuing the work of Reich and Sabach \cite{RS1}, Ogbuisi and Izuchukwu \cite{OI} studied the following inclusion problem in a reflexive Banach space $\mathbb{X}$:
\begin{eqnarray}\label{MVIP}
\mbox{Find} ~v\in \mathbb{X} ~\mbox{such that}~ 0^* \in  \left(\mathcal{A}+ \mathcal{B}\right)v,
\end{eqnarray}
where  $\mathcal{A}:\mathbb{X}\rightarrow \mathbb{X}^*$ is a single-valued operator and $\mathcal{B}$ is as defined in \eqref{NPP}. \\
If the operators are monotone, then Problem \eqref{MVIP} can be referred to as a Monotone Inclusion Problem (MIP). It is worth mentioning that if $\mathcal{B}=N_\mathcal{C}$ in Problem \eqref{MVIP}, where $N_\mathcal{C}$ is the normal cone associated with a nonempty, closed and convex subset $\mathcal{C}$ of $\mathbb{X}$, then Problem \eqref{MVIP} reduces to the classical variational inequality problem.
Also, when $\mathcal{A}= 0$, we recover the NPP \eqref{NPP} as a special case of Problem \eqref{MVIP}. Furthermore, several problems in signal processing, image recovery and machine learning can be formulated as Problem \eqref{MVIP}. Therefore, we see that Problem \eqref{MVIP} is quite general; it naturally includes many other important optimization problems such as minimization problems, linear inverse problems, saddle-point problems, fixed point problems, split feasibility problems, Nash equilibrium problems in noncooperative games, and many more. Several authors have studied Problem \eqref{MVIP} using the well-known forward-backward splitting method and its modifications in real Hilbert spaces (see, for instance, \cite{AC, BauschkePro, Bui, 4SS, 5SS, Tseng, 11SS, Malitsky2}).\\
In order to solve Problem \eqref{MVIP} in reflexive Banach spaces, Ogbuisi and Izuchukwu \cite{OI} proposed the following iterative method:
\begin{eqnarray}\label{3*}
\begin{cases}
u, x_1\in C,~~C_1= C\\
y_n=\nabla g^*\left(\alpha_n \nabla g(u) +\beta_n\nabla g(x_n)+\gamma_n \nabla g(T (x_n ))\right)\\
u_n=(\mbox{Res}^g_{\mu \mathcal{B}}\circ A^g_\mu)y_n\\
C_{n+1}=\{z\in C_n: D_g(z, u_n)\leq  \alpha_n D_g(z, u)+ (1-\alpha_{n}) D_g(z, x_n)\}\\
x_{n+1}=\mbox{Proj}^g_{C_{n+1}}(x_1), ~n\geq 1,
\end{cases}
\end{eqnarray}
where $\mu>0$, $\{\alpha_n\},~\{\beta_n\}$ and $\{\gamma_n\}$ are sequences in $(0, 1)$, $A_{\mu}^g:=\nabla g^*\circ (\nabla g-\mu \mathcal{A})$ and $T$ is a Bregman strongly nonexpansive mapping. \\
These authors proved that the sequence generated by Algorithm \eqref{3*} converges to a common solution of Problem \eqref{MVIP} and the fixed point problem for the mapping $T$, provided that the solution set of the problem is nonempty, $\mathcal{A}$ is Bregman inverse strongly monotone, $\mathcal{B}$ is maximal monotone,
$\lim_{n\to\infty} \alpha_n =0,~\alpha_n +\beta_n+\gamma_n=1$, and $0<a<\beta_n,\gamma_n<b<1$. \\
 Recently, Chang {\it et al.} \cite{SS} proposed a method similar to \eqref{3*} (see \cite[Algorithm (11)]{SS} for solving Problem \eqref{MVIP} when $\mathcal{A}$ is Bregman inverse strongly monotone and $\mathcal{B}$ is maximal monotone. Other recent methods for solving the MIP \eqref{MVIP}, under the same assumptions on $\mathcal{A}$ and  $\mathcal{B}$ in reflexive Banach spaces, can be found in \cite{Ogbuisi, CI1, BEDRP, Tang, Tuyen}.

\noindent Very recently, Sunthrayuth {\it et al.} \cite{Sun} proposed the following  modification of Tseng's splitting method \cite{Tseng} for solving Problem \eqref{MVIP} in a reflexive Banach space when $\mathcal{A}$ is monotone and Lipschitz continuous, and $\mathcal{B}$ is maximal monotone:
\begin{eqnarray}\label{Sun}
\begin{cases}
x_1\in \mathbb{X}\\
y_n=\mbox{Res}^g_{\mu_n \mathcal{B}} \nabla g^*\left(\nabla g (x_n)-\mu_n \mathcal{A} x_n\right)\\
z_n=\nabla g^*\left(\nabla g (y_n)-\mu_n (\mathcal{A}y_n-\mathcal{A} x_n)\right)\\
x_{n+1}=\nabla g^*\left((1-\alpha_n)\nabla g (z_n)+\alpha_n \nabla g(Tz_n))\right), ~n\geq 1,
\end{cases}
\end{eqnarray}
where $\{\alpha_n\}$ is a sequence in $(0, 1)$, $\mu_n=\gamma l^{m_n}$ and $m_n$ is the smallest nonnegative integer such that
\begin{eqnarray}\label{LS}
\mu_n||\mathcal{A}x_n-\mathcal{A}y_n||\leq \mu ||x_n-y_n||,
\end{eqnarray}
with $\gamma>0, l\in (0, 1), \mu \in (0, \alpha)$ and  $\alpha>0$.\\
Furthermore, they  proposed a variant of \eqref{Sun}--\eqref{LS} by replacing only \eqref{LS} with a self-adaptive step size procedure (see \cite[Algorithm 2]{Sun}) for solving Problem \eqref{MVIP} under the same assumptions on $\mathcal{A}$ and $\mathcal{B}$ in a reflexive Banach space.\\
At this point we mention another modification of Tseng's splitting method due to Shehu \cite{Shehu} for solving the MIP \eqref{MVIP} in the context of $2$-uniformly convex and uniformly smooth Banach spaces, when $\mathcal{A}$ is monotone and Lipschitz continuous, and $\mathcal{B}$ is maximal monotone.\\
Unfortunately, all these methods for solving the MIP \eqref{MVIP}, when $\mathcal{A}$ is monotone and Lipschitz continuous in reflexive Banach spaces, require at each iteration at least two evaluations of $\mathcal{A}$ and this might affect the efficiency of these methods especially in situations where evaluating the operator $\mathcal{A}$ is expensive, for instance, in problems arising from optimal control.

\noindent In this paper, inspired by the results in \cite{SS, OI, Malitsky2, Shehu, Sun}, we propose two methods for solving the MIP \eqref{MVIP} in real reflexive Banach spaces.  Our methods have simple and elegant structures, and they only require one evaluation of $\mathcal{A}$ at each iteration.
We obtain weak convergence results when $\mathcal{B}$ is maximal monotone and $\mathcal{A}$ is Lipschitz continuous, and we obtain strong convergence results when either $\mathcal{A}$ or $\mathcal{B}$ is required, in addition, to be strongly monotone. We also obtain the rate of convergence of our proposed methods in real reflexive Banach spaces and apply our results to solving generalized Nash equilibrium problems for gas markets.

\noindent The rest of our paper is organized as follows:  Section \ref{Se2} contains basic definitions and results. In Section \ref{Se3}, we present and discuss the methods we propose. In Section \ref{Se4}, we obtain some convergence results for our methods. As special cases, we also obtain methods for solving the NPP \eqref{NPP} and the variational inequality problem.
In Section \ref{Se5}, we consider a generalized Nash equilibrium approach for modelling gas markets and apply our results to solving this problem.
We then give some concluding remarks in Section \ref{Se6}.

\section{Preliminaries}\label{Se2}
\noindent Let $\mathbb{X}$ be a real reflexive Banach space and let $\mathbb{X}^*$ be its dual. Let $g:\mathbb{X}\to(-\infty, +\infty]$ be a function. The domain of $g$, denoted by ${dom}_g$, is defined by ${dom}_g := \{x\in \mathbb{X}: g(x)<+\infty\}$.
 The function $g:\mathbb{X}\to (-\infty, +\infty]$ is called proper if ${dom}_g\neq \emptyset$, and convex if $g(\lambda x+(1-\lambda) y)\leq \lambda g(x)+(1-\lambda)g(y) \; ~\forall x, y\in \mathbb{X},~\lambda\in (0, 1)$.
A function $g:{dom}_g\subseteq \mathbb{X}\to(-\infty, \infty]$ is said to be lower semicontinuous at a point $x\in {dom}_g$, if $ g(x)\leq \liminf_{n\to\infty} g(x_n)$ for each sequence $\{x_n\}$ in ${dom}_g$ such that $\underset{n\to\infty}\lim x_n=x$. It is is said to be lower semicontinuous on ${dom}_g$ if it is lower semicontinuous at each point in ${dom}_g$.

\noindent The Fenchel conjugate of $g$ is the function $g^*:\mathbb{X}^*\to (-\infty, \infty]$ defined by $$g^*(x^*) := \sup\{\langle x^*, x\rangle-g(x): x\in \mathbb{X}\}.$$
Let $x\in$ $int({dom}_g)$, where $int({dom}_g)$ stands for the interior of the domain of $g$. Then for any $y\in \mathbb{X}$, we define the right-hand derivative of $g$ at $x$ by
\begin{eqnarray}\label{GD}
g' (x, y):=\lim_{\lambda\to 0^+} \frac{g(x+\lambda y)-g(x)}{\lambda}.
\end{eqnarray}
The function $g$ is said to be G\^{a}teaux differentiable at $x$ if the limit in \eqref{GD} exists as $\lambda\to 0$ for each $y\in \mathbb{X}$. In this case, the gradient of $g$ at $x$ is the linear function $\nabla g(x)$, defined by $\langle \nabla g(x), y\rangle:=g' (x,y)~ \forall y\in \mathbb{X}$. We say that $g$ is G\^{a}teaux differentiable if it is G\^{a}teaux differentiable at each $x\in$ $int({dom}_g)$.
If the limit in \eqref{GD} is attained uniformly for $y\in \mathbb{X}$ with $||y||=1$, we say that $g$ is Fr\'echet differentiable at $x$. If the limit in \eqref{GD} is attained uniformly for $x\in \mathcal{C}$ and for $y\in \mathbb{X}$ with $||y||=1$, then we say that the function $g$ is uniformly Fr\'echet differentiable on the subset $\mathcal{C}$ of $\mathbb{X}$.\\
The function $g$ is said to be Legendre if it satisfies the following two conditions.
 \begin{enumerate}
 \item[(i)] $g$ is G\^{a}teaux differentiable, $int({dom}_g)\neq \emptyset$ and ${dom}_{\nabla g}=$ $int({dom}_g)$;
 \item[(ii)] $g^*$ is G\^{a}teaux differentiable,  $int({dom}_{g^*})\neq \emptyset$ and  ${dom}_{\nabla g^*}=$ $int({dom}_{g^*})$.
 \end{enumerate}
It is well known that $\nabla g=(\nabla g^*)^{-1}$ in reflexive Banach spaces. Combining this fact with conditions (i) and (ii), we get that
${ran}_{\nabla g}={dom}_{\nabla g^*}=$ $int({dom}_{g^*})$ and ${ran}_{\nabla g^*}={dom}_{\nabla g}=int({dom}_g)$, where ${ran}_{g}$ denotes the range of $g$. \\
We also know that conditions (i) and (ii) imply that the functions $g$ and $g^*$ are G\^{a}teaux differentiable and strictly convex in the interior of their respective domains. Hence, $g$ is Legendre if and only if $g^*$ is Legendre.

\noindent  The bifunction $D_g:{dom}_g\times {int({dom}_g)}\to [0, +\infty)$, defined by
 \begin{eqnarray}
 D_g(x, y):=g(x)-g(y)-\langle \nabla g(y), x-y\rangle,
 \end{eqnarray}
 is called a Bregman distance. If $g$ is a G\^{a}teaux differentiable function, then the Bregman distance has the following important property, called the three point identity:  for any $x\in {dom}_g$ and $y, z\in {int({dom}_g)}$,
 \begin{eqnarray}\label{3P}
 D_g(x,y)+D_g(y, z)-D_g(x, z)=\langle \nabla g(z)-\nabla g(y), x-y\rangle.
 \end{eqnarray}
\noindent A G\^{a}teaux differentiable function $g$ is called strongly convex (see \cite{31LY, 32LY}), if there exists $\gamma>0$ such that $\langle \nabla g(x)-\nabla g(y), x-y\rangle\geq \gamma ||x-y||^2$  or equivalently, $g(y)\geq g(x)+\langle \nabla g(x), y-x\rangle+\frac{\gamma}{2}||x-y||^2 \; ~\forall x,y\in {dom}_g$.
\begin{remark}
If $g$ is a strongly convex function with constant $\gamma>0$, then
\begin{eqnarray}\label{17AA}
D_g(x, y)\geq \frac{\gamma}{2}||x-y||^2~\forall x\in {dom}_g,~y\in {int({dom}_g)}.
\end{eqnarray}
\end{remark}
\noindent The function $g:\mathbb{X}\to \mathbb{R}\cup\{+\infty\}$ is said to be strongly coercive if
$$\lim_{||x||\to\infty}\frac{g(x)}{||x||}=+\infty.$$
\begin{remark} \label{CONT} (see \cite{33LY}).
If $g:\mathbb{X}\to \mathbb{R}$ is strongly coercive, then
\begin{enumerate}
\item[(i)] $\nabla g:\mathbb{X}\to \mathbb{X}^*$ is one-to-one, onto and norm-to-weak* continuous;
\item[(ii)] $\{x\in \mathbb{X} : D_g(x, y)\leq r\}$ is bounded for all $y\in \mathbb{X}$ and $r>0$;
\item[(iii)] ${dom}_{g^*}=\mathbb{X}^*$, $g^*$ is G\^{a}teaux differentiable and $\nabla g^*=(\nabla g)^{-1}$.
\end{enumerate}
\end{remark}
\noindent The operator $\mathcal{A}: \mathbb{X}\rightarrow \mathbb{X}^*$ is said to be $L$-{\it Lipschitz continuous} if there exists $L>0$ such that
	\begin{align*}
	\|\mathcal{A}x-\mathcal{A}y\|\leq L\|x-y\|~ \forall x,y\in \mathbb{X}.
	\end{align*}
$A$ is called {\it $\tau$-strongly monotone} if there exists  $\tau>0$ such that
	\begin{align*}
	\langle \mathcal{A}x-\mathcal{A}y,x-y\rangle\geq \tau\|x-y\|^2~ \forall x,y \in \mathbb{X},
	\end{align*}
and {\it monotone} if
	\begin{align*}\langle \mathcal{A}x-\mathcal{A}y,x-y\rangle\geq 0 \; ~ \forall x,y \in \mathbb{X}.
	\end{align*}
\noindent If $\mathcal{B}$ is a set-valued operator, that is, $\mathcal{B}:\mathbb{X}\rightarrow 2^{\mathbb{X}^*}$, then $\mathcal{B}$ is called
{\it $\tau$-strongly monotone} if there exists  $\tau>0$ such that
	\begin{align*}
	\langle u-v,x-y\rangle\geq \tau\|x-y\|^2~\forall x,y \in \mathbb{X},~u\in \mathcal{B}x,~v\in \mathcal{B}y,
	\end{align*}
and  {\it monotone} if $$\langle u-v, x-y\rangle\geq 0 \; ~\forall x,y \in \mathbb{X},~u\in \mathcal{B}x,~v\in \mathcal{B}y.$$ The monotone operator $\mathcal{B}$ is said to be {maximal} if the graph $G(\mathcal{B})$ of $\mathcal{B}$, defined by
$$G(\mathcal{B}):=\{(x, y)\in \mathbb{X}\times \mathbb{X^*} : y\in \mathcal{B}x\},$$ is not properly contained in the graph of any other monotone operator.
In other words, $\mathcal{B}$ is maximal monotone if and only if for $(x,u)\in \mathbb{X}\times \mathbb{X}^*$, the assumption that $\langle u-v, x-y\rangle \geq 0$ for all $(y,v)\in G(\mathcal{B})$ implies that $u\in \mathcal{B}x$. It is known \cite[Corollary 2.4]{7Sun} that if $g:\mathbb{X}\to \mathbb{R}$ is G\^{a}teaux differentiable, strictly convex and cofinite, then $\mathcal{B}$ is maximal monotone if and only if ${ran}_{(\nabla g +\mu \mathcal{B})}=\mathbb{X}^*$.  \\
Let $g:\mathbb{X}\to \mathbb{R}$ be a G\^{a}teaux differentiable function on $\mathbb{X}$. Then the resolvent operator $\mbox{Res}_{\mu \mathcal{B}}^g$ associated with a set-valued operator $\mathcal{B}$ and $\mu>0$, and relative to $g$, is the mapping $\mbox{Res}_{\mu \mathcal{B}}^g: \mathbb{X}\rightarrow 2^{\mathbb{X}}$ defined by
\begin{eqnarray}\label{RES}
\mbox{Res}_{\mu \mathcal{B}}^g := (\nabla g+\mu \mathcal{B})^{-1} \circ \nabla g.
\end{eqnarray}
The resolvent operator is single-valued when $\mathcal{B}$ is monotone and $g$ is strictly convex on $int({dom}_g)$.

\noindent Let $\mathcal{C}$ be a nonempty, closed and convex subset of a reflexive Banach space $\mathbb{X}$. The mapping $\mathcal{A}:\mathbb{X}\to 2^{\mathbb{X}^*}$ is said to be Bregman inverse strongly monotone on the set $\mathcal{C}$ if $\mathcal{C}\cap ({dom}_g)\cap int({dom}_g)\neq \emptyset$ and for any $x, y\in \mathcal{C}\cap int ({dom}_g)$, $\xi \in \mathcal{A}x$ and $\eta \in \mathcal{A}y$, we have $\langle \xi-\eta, \nabla g^*(\nabla g(x)-\xi)-\nabla g^*(\nabla g(y)-\eta)\rangle\geq 0.
$

\noindent Let $g:\mathbb{X}\to \mathbb{R}\cup\{+\infty\}$ be a convex and G\^{a}teaux differentiable function, and let $\mathcal{C}$ be a nonempty, closed and convex subset of a  real reflexive Banach space $\mathbb{X}$. The Bregman projection of $x\in int({dom}_g)$ onto $\mathcal{C}\subset int({dom}_g)$ is the unique vector ${Proj}^g_\mathcal{C}(x)\in \mathcal{C}$ satisfying (see \cite{BREG}) $D_g \left({Proj}^g_\mathcal{C} (x), x\right)=\inf \{ D_g (y, x):y\in \mathcal{C}\}.$\\
\noindent The normal cone of $\mathcal{C}$ at a point $z\in \mathbb{X}$ is defined by
\begin{eqnarray}\label{NC}
N_\mathcal{C} z:=\{d^*\in \mathbb{X}^*: \langle d^*, y-z\rangle \leq 0 \; ~~\forall y\in \mathcal{C}\} \; ~\mbox{if}~z\in \mathcal{C}~~\mbox{and}~\emptyset, ~\mbox{otherwise}.
\end{eqnarray}

\begin{lemma}\cite{11Cholam}\label{2.17}
Let $g:\mathbb{X}\to \mathbb{R}$ be a proper, convex, lower semicontinuous  and G\^{a}teaux differentiable on $int({dom}_g)$ such that $\nabla g^*$ is bounded on bounded subsets of ${dom}_{g^*}$. Let $x^*\in \mathbb{X}$ and $\{x_n\}\subset$ $\mathbb{X}$. If $\{D_g(x, x_n)\}$ is bounded, so is the sequence $\{x_n\}$.
\end{lemma}

\begin{lemma}\cite{35LY}\label{Opial}
Let $\mathbb{X}$ be a Banach space, and let $g:\mathbb{X}\to(-\infty, \infty]$ be a proper and strictly convex function such that $g$ is G\^{a}teaux differentiable. Suppose that there exists a point $p\in \mathbb{X}$ such that the sequence $\{x_n\}$ in $\mathbb{X}$ converges weakly to $p$. Then
$$\limsup_{n\to\infty} D_g(p, x_n)<\limsup_{n\to\infty} D_g(q, x_n)$$ for all $q\in {int({dom}_g)}$ with $p\neq q$.
\end{lemma}

\begin{lemma}\label{le5} \cite[Corollary 2.1 and Theorem 1.7]{4Sun}
Let $\mathcal{A}:\mathbb{X}\rightarrow \mathbb{X}^*$ be a monotone and  Lipschitz continuous operator, and let $\mathcal{B}:\mathbb{X}\rightarrow 2^{\mathbb{X}^*}$ be a maximal monotone operator. Then $\mathcal{A}+\mathcal{B}$ is maximal monotone.
\end{lemma}

%\begin{lemma}\cite[Proposition 9]{30Sun}\label{Opial}
%Let $g:\mathbb{X}\to\mathbb{R}$ be a Legendre function such that $\nabla g$
%is weakly sequentially continuous. Suppose that the sequence $\{v_n\}$ is bounded and that $\lim_{n\to \infty} D_g (u, v_n)$ exists for any weak subsequential limit $u$ of $\{v_n\}$. Then $\{v_n\}$ converges weakly to $u$.
%\end{lemma}

\section{Proposed Methods}\label{Se3}
\noindent  In this section we present our proposed methods for solving the MIP \eqref{MVIP}. Throughout this section we assume that the solution set of \eqref{MVIP} is nonempty, that is, $(\mathcal{A}+\mathcal{B})^{-1}(0^*)\neq \emptyset$. We also assume for the rest of this paper that the function $g:\mathbb{X}\to \mathbb{R}\cup \{+\infty\}$ is proper, convex, lower semicontinuous, uniformly Fr\'echet  differentiable, $\gamma$-strongly convex, strongly coercive and Legendre. In addition, we make the following assumptions.

\hrule\hrule	
	\begin{ass}\label{ass1} Let $\mathbb{X}$ be a real reflexive Banach space $\mathbb{X}$, let $\mathbb{X}^*$ denote its dual, and let $\mathcal{A}:\mathbb{X}\to \mathbb{X}^*$ and $\mathcal{B}:\mathbb{X}\to 2^{\mathbb{X}^*}$ be two operators satisfying the following conditions:
\hrule\hrule
\begin{itemize}
\item[(a)] $\mathcal{A}$  is monotone on $\mathbb{X}$ and $\mathcal{B}$ is maximal monotone on $\mathbb{X}$,
\item[(b)]  $\mathcal{A}$ is Lipschitz continuous on $\mathbb{X}$ with constant $L>0$.
		\end{itemize}
\end{ass}

\noindent When $L$ is known, we present the following method for solving the inclusion problem \eqref{MVIP}.
\hrule \hrule
\begin{alg}\label{al1} For arbitrary $v_0,v_1 \in \mathbb{X}$ and $\mu>0$, define the sequence $\{v_n\}$ by
	\hrule \hrule
	 \begin{eqnarray*}\label{a1}
v_{n+1}={Res}^g_{\mu\mathcal{B}}\Big(\nabla g^*(\nabla g(v_n)-\mu(2\mathcal{A}v_n-\mathcal{A}v_{n-1}))\Big),~ n\geq 1.
\end{eqnarray*}
\hrule\hrule
\end{alg}

\hfill

\noindent  When $L$ is unknown, we present the method below with self-adaptive step sizes for solving the inclusion problem \eqref{MVIP}.
\hrule \hrule
\begin{alg}\label{AL1A} Let $\alpha\in (0, 1)$, $\mu_0,\mu_1>0$ and choose a nonnegative real sequence $\{d_n\}$  such that $\sum_{n=1}^{\infty} d_n< \infty.$ For arbitrary $v_0,v_1 \in \mathbb{X}$, let the sequence $\{v_n\}$ be generated by
	\hrule \hrule
	 \begin{eqnarray*}\label{a1A}
v_{n+1}={Res}^g_{\mu_n\mathcal{B}}\Big(\nabla g^*(\nabla g(v_n)-\left((\mu_n+\mu_{n-1})\mathcal{A}v_n-\mu_{n-1}\mathcal{A}v_{n-1}\right))\Big),~ n\geq 1,
\end{eqnarray*}
where
\begin{eqnarray}\label{eq1A}
	\mu_{n+1}=\begin{cases}
	\min \left\{\frac{\alpha\|v_n-v_{n+1}\|_{\mathbb{X}}}{\|\mathcal{A}v_n-\mathcal{A}v_{n+1}\|_{\mathbb{X}^*}},~\mu_{n}+d_{n}\right\}, & \mbox{if}~  \mathcal{A}v_n\neq \mathcal{A}v_{n+1},\\
	\mu_{n}+d_n,& \mbox{otherwise}.
	\end{cases}
	\end{eqnarray}
\hrule\hrule
\end{alg}

\begin{remark} \label{VSS} ${}$
\begin{itemize}
\item Clearly, Algorithms \ref{al1} and \ref{AL1A} only require one evaluation of  $\mathcal{A}$ per iteration, unlike the methods in \cite{LY,  Yekini, Shehu,  Sun} which require two evaluations of $\mathcal{A}$ per iteration.

\item Note that by \eqref{eq1A}, $\lim\limits_{n\to\infty}\mu_n=\mu$, where $\mu\in [\min\{\frac{\alpha}{L}, \mu_1\}, ~\mu_1+d]$ with $d=\sum_{n=1}^\infty d_n$  (see \cite{HY}).

\item When $d_n=0$, then the step size $\mu_n$ in \eqref{eq1A} is similar to the one in \cite{HP, LY, Sun}. We recall that the step size in \cite{HP, LY, Sun} is monotonically decreasing; so their methods may depend on the choice of the initial step size $\mu_1$. However, the step size  given in \eqref{eq1A} is non-monotonic and so the dependence on the initial step size $\mu_1$ is reduced.
 \end{itemize}
\end{remark}

\section{Convergence Results}\label{Se4}
\subsection{Weak Convergence} ${}$\\
\noindent In this subsection we consider the weak convergence of the sequences generated by Algorithms \ref{al1} and \ref{AL1A}. We begin with those
generated by Algorithm \ref{al1}.

\begin{lemma}\label{lem1} Let $\{v_n\}$ be generated by Algorithm \ref{al1} when Assumption \ref{ass1}(a) holds. Then
	\begin{align*}
	D_g(z,v_{n+1})&\leq D_g(z,v_n)+\mu\langle \mathcal{A}v_n-\mathcal{A}v_{n-1},z-v_n\rangle+\mu\langle \mathcal{A}v_n-\mathcal{A}v_{n-1},v_n-v_{n+1}\rangle\\
	&+\mu\langle \mathcal{A}v_n-\mathcal{A}v_{n+1}, z-v_{n+1}\rangle -D_g(v_{n+1},v_n) \hspace{0.2cm}\forall z\in (\mathcal{A}+\mathcal{B})^{-1}(0^*).
	\end{align*}
\end{lemma}
	\begin{proof}
		Let $z\in (\mathcal{A}+\mathcal{B})^{-1}(0^*).$  Then $-\mathcal{A}z\in \mathcal{B}z.$
		
		\noindent By  the definitions of $v_{n+1}$ and ${Res}^g_{\mu\mathcal{B}}$ (see \eqref{RES}), we obtain
		$$v_{n+1}=(\nabla g +\mu \mathcal{B})^{-1}\left(\nabla g(v_n)-\mu(2\mathcal{A}v_n-\mathcal{A}v_{n-1})\right),$$
		which implies that
		\begin{eqnarray}\label{**}
		\frac{1}{\mu}\left(\nabla g(v_n)-\mu (2\mathcal{A}v_n-\mathcal{A}v_{n-1})-\nabla g(v_{n+1})\right)\in B v_{n+1}.
		\end{eqnarray}
		Hence, using the monotonicity of $\mathcal{B}$, we get
		$$0\leq \left\langle \frac{1}{\mu}\left(\nabla g (v_n)-\mu (2\mathcal{A}v_n-\mathcal{A}v_{n-1})-\nabla  g(v_{n+1})\right)+\mathcal{A}z, v_{n+1}-z\right\rangle,$$
	which implies that
	\begin{eqnarray}\label{M1}
	0&\leq & \langle \nabla g (v_{n+1})-\nabla g (v_n)+\mu (2\mathcal{A}v_n-\mathcal{A}v_{n-1})-\mu\mathcal{A}z, z-v_{n+1}\rangle\nonumber\\
	&=&\langle \nabla g (v_{n+1})-\nabla g (v_n), z-v_{n+1}\rangle+\mu \langle \mathcal{A}v_n-\mathcal{A}z, z-v_{n+1}\rangle+\mu \langle \mathcal{A}v_n-\mathcal{A}v_{n-1}, z-v_{n+1}\rangle.
\end{eqnarray}		
Next, using the monotonicity of $\mathcal{A}$, we see that
\begin{eqnarray}\label{M2}
\langle \mathcal{A}v_n-\mathcal{A}z, z-v_{n+1}\rangle\leq \langle \mathcal{A}v_n-\mathcal{A}v_{n+1}, z-v_{n+1}\rangle.
\end{eqnarray}
Now, using \eqref{M2} in \eqref{M1}, and noting equation \eqref{3P}, we obtain
		\begin{align}\label{a2}\nonumber
		0&\leq \langle \nabla g(v_{n+1})-\nabla g(v_n),z-v_{n+1}\rangle +\mu\langle \mathcal{A}v_n-\mathcal{A}v_{n+1},z-v_{n+1}\rangle +\mu\langle \mathcal{A}v_n-\mathcal{A}v_{n-1},z-v_{n}\rangle \\ \nonumber
		& +\mu\langle \mathcal{A}v_n-\mathcal{A}v_{n-1},v_n-v_{n+1}\rangle \\ \nonumber
		&=D_g(z,v_n)-D_g(z,v_{n+1})-D_g(v_{n+1}, v_n)+\mu\langle \mathcal{A}v_n-\mathcal{A}v_{n+1},z-v_{n+1}\rangle +\mu\langle \mathcal{A}v_n-\mathcal{A}v_{n-1},z-v_{n}\rangle \\
		& +\mu\langle \mathcal{A}v_n-\mathcal{A}v_{n-1},v_n-v_{n+1}\rangle,
		\end{align}
which yields the required conclusion.
				\end{proof}

		\begin{theorem}\label{thm 1}
Let  Assumption \ref{ass1} hold and let $\mu\in \Big[\delta,~\frac{\gamma(1-2\delta)}{2L}\Big]$ for some $\delta\in (0, \frac{1}{2})$ and $~ \gamma>0.$ Then the sequence $\{v_n\}$ generated by Algorithm \ref{al1} converges weakly to an element of $(\mathcal{A}+\mathcal{B})^{-1}(0^*).$
		\end{theorem}
	\begin{proof}
				Using the Lipschitz continuity of $\mathcal{A}$ and \eqref{17AA}, we obtain
	\begin{align}\label{a3}\nonumber
	\mu\langle \mathcal{A}v_n-\mathcal{A}v_{n-1},v_n-v_{n+1}\rangle&\leq \mu L\|v_n-v_{n-1}\|\|v_n-v_{n+1}\|\\ \nonumber
	&\leq \frac{\mu L}{2}\Big(\|v_n-v_{n-1}\|^2+\|v_{n+1}-v_n\|^2\Big)\\
	&\leq \mu L \gamma^{-1}\Big(D_g(v_n,v_{n-1})+D_g(v_{n+1},v_n)\Big).
	\end{align}
Since $\mu\leq\frac{\gamma(1-2\delta)}{2L},$ we get that $\mu L\gamma^{-1}\leq \frac{1}{2}-\frac{2\delta}{2}<\frac{1}{2},$ which further gives that $\mu L\gamma^{-1}-1\leq -\left(\frac{1}{2}+\delta\right)$. Using these inequalities and \eqref{a3} in Lemma \ref{lem1}, we see that
	\begin{align}\label{a5}\nonumber
	D_g(z,v_{n+1})&\leq D_g(z,v_n)+\mu\langle \mathcal{A}v_n-\mathcal{A}v_{n-1},z-v_n\rangle +\mu L\gamma^{-1}D_g(v_n,v_{n-1})\\ \nonumber
	&+\mu \langle \mathcal{A}v_n-\mathcal{A}v_{n+1},z-v_{n+1}\rangle +(\mu L\gamma^{-1}-1)D_g(v_{n+1},v_n)\\ \nonumber
	&\leq D_g(z,v_n)+\mu\langle \mathcal{A}v_n-\mathcal{A}v_{n-1},z-v_n\rangle +\frac{1}{2}D_g(v_n,v_{n-1})\\
				&+\mu\langle \mathcal{A}v_n-\mathcal{A}v_{n+1},z-v_{n+1}\rangle -\Big(\frac{1}{2}+\delta\Big)D_g(v_{n+1},v_n).
				\end{align}
Now, for $n\geq 1$, let
\begin{eqnarray*}
s_n&=&D_g(z,v_n)+\mu\langle \mathcal{A}v_n-\mathcal{A}v_{n-1},z-v_n\rangle +\frac{1}{2}D_g(v_n,v_{n-1}),\\
t_n&=&\delta D_g(v_{n+1},v_n).
\end{eqnarray*}
Then \eqref{a5} can be rewritten as
\begin{eqnarray}\label{4.8*}
s_{n+1}\leq s_{n}-t_n~\forall n\geq 1.
\end{eqnarray}
We have $s_n\geq 0 \; ~\forall n\geq 1$, because
\begin{eqnarray}\label{4.8**}
s_n&= &D_g(z,v_n)+\mu\langle \mathcal{A}v_n-\mathcal{A}v_{n-1},z-v_n\rangle +\frac{1}{2}D_g(v_n,v_{n-1}) \nonumber\\
&\geq &D_g(z,v_n)-\mu L \gamma^{-1}\left(D_g(v_n, v_{n-1})+D_g(z,v_n)\right)+\frac{1}{2}D_g(v_n,v_{n-1}) \nonumber\\
&\geq &D_g(z,v_n)-\frac{1}{2}\left(D_g(v_n, v_{n-1})+D_g(z,v_n)\right)+\frac{1}{2}D_g(v_n,v_{n-1}) \nonumber\\
&=&\frac{1}{2}D_g(z,v_n)\geq 0.
\end{eqnarray}
Hence, \eqref{4.8*} implies that the sequence $\{s_n\}$ is bounded and that $\lim\limits_{n\rightarrow \infty} t_n=0.$ Since $\{s_n\}$ is bounded, it follows from \eqref{4.8**} that  $\{D_g(z,v_n)\}$ is also bounded, and using Lemma \ref{2.17}, we get that the sequence $\{v_n\}$ is bounded. Let $u$ be a weak cluster point of $\{v_n\}.$
Then we can choose a subsequence of $\{v_n\},$ denoted by $\{v_{n_j}\}$, such that $v_{n_j}\rightharpoonup u.$
Once again,	since  $\lim\limits_{n\rightarrow \infty}D_g(v_{n+1},v_n)=\frac{1}{\delta}\lim\limits_{n\rightarrow \infty}t_n=0,$  we obtain that $\lim\limits_{n\rightarrow \infty}\|v_{n+1}-v_n\|=0$ (by inequality \eqref{17AA}). Since $g$ is strongly coercive, we get that $\lim\limits_{n\rightarrow \infty}\|\nabla g(v_{n+1})-\nabla g(v_n)\|=0$ (because if $g$ is uniformly Fr\'echet differentiable, then $\nabla g^*$ is uniformly continuous on bounded subsets of $\mathbb{X}^*$). Also, since $\mathcal{A}$ is Lipschitz continuous (hence uniformly continuous), we obtain that $\lim\limits_{n\rightarrow \infty}\|\mathcal{A}v_n-\mathcal{A}v_{n-1}\|=0.$	

\noindent Now, consider $(v, w)\in G(\mathcal{A}+\mathcal{B})$. Then $w-\mathcal{A}v\in \mathcal{B}v$. Using this, \eqref{**} and the monotonicity of $\mathcal{B}$, we get that
\begin{eqnarray}\label{a12}
\langle w - \mathcal{A}v-\frac{1}{\mu}\left(\nabla g (v_{n_j})-\mu (2\mathcal{A}v_{n_j}-\mathcal{A}v_{n_j-1})-\nabla g(v_{n_j+1})\right), v - v_{n_j+1}\rangle \geq 0.
\end{eqnarray}
Using  \eqref{a12} and the monotonicity of $\mathcal{A}$, we obtain that
\begin{eqnarray}\label{a13}
\langle w, v-v_{n_j+1}\rangle &\geq& \langle \mathcal{A}v+\frac{1}{\mu}\left(\nabla g(v_{n_j})-\mu(2\mathcal{A}v_{n_j}-\mathcal{A}v_{n_j-1})-\nabla g(v_{n_j+1})\right), v - v_{n_j+1}\rangle\nonumber\\
&=& \langle \mathcal{A}v-\mathcal{A}v_{n_j+1}, v - v_{n_j+1}\rangle + \langle \mathcal{A}v_{n_j+1}-\mathcal{A}v_{n_j}, v - v_{n_j+1}\rangle\nonumber\\
&& +\langle \mathcal{A}v_{n_j-1}-\mathcal{A}v_{n_j}, v - v_{n_j+1}\rangle + \frac{1}{\mu}\langle \nabla g(v_{n_j})-\nabla g(v_{n_j+1}), v - v_{n_j+1}\rangle\nonumber\\
&\geq &   \langle \mathcal{A}v_{n_j+1}-\mathcal{A}v_{n_j}, v - v_{n_j+1}\rangle +\langle \mathcal{A}v_{n_j-1}-\mathcal{A}v_{n_j}, v - v_{n_j+1}\rangle\nonumber\\
&&  + \frac{1}{\mu}\langle \nabla g(v_{n_j})-\nabla g(v_{n_j+1}), v - v_{n_j+1}\rangle.
\end{eqnarray}
Passing to the limit as $j\to \infty$ in \eqref{a13}, we see that $\langle w, v-u\rangle\geq 0.$ Thus, using the maximal monotonicity of $\mathcal{A}+\mathcal{B}$ (see Lemma \ref{le5}), we conclude that $u\in (\mathcal{A}+\mathcal{B})^{-1}(0^*)$.

\noindent We now show that $\{v_n\}$ converges weakly to $u$. It follows from \eqref{4.8*} that the sequence $\{s_n\}$ is monotone for all $z\in (\mathcal{A}+\mathcal{B})^{-1}(0^*)$, and since it is also bounded for all $z\in (\mathcal{A}+\mathcal{B})^{-1}(0^*)$, we get that
\begin{eqnarray}\label{***}
\lim_{n\to\infty}\left(D_g(u, v_n)+\mu\langle \mathcal{A}v_n-\mathcal{A}v_{n-1}, u-v_n\rangle+\frac{1}{2}D_g(v_n, v_{n-1})\right)~~\mbox{exists}.
\end{eqnarray}
Since $\{v_n\}$ is bounded, $D_g(v_{n},v_{n-1})\to 0$ as $n\to \infty$, and $\mathcal{A}$ is Lipschitz continuous, it follows from  \eqref{***} that $\lim\limits_{n\to\infty} D_g(u, v_n)$ exists.\\
Next, we show that $u$ is unique. Suppose to the contrary that this is not true. Then there exists a subsequence $\{v_{n_i}\}\subset\{v_n\}$ such that $v_{n_i}\rightharpoonup \bar{u}$ with $u\neq \bar{u}$. Note that we can again show that $\bar{u}\in (\mathcal{A}+\mathcal{B})^{-1}(0^*)$. Using Lemma \ref{Opial}, we get
\begin{eqnarray*}
\lim_{n\to \infty} D_g(u, v_n)&=&\limsup_{j\to\infty}D_g(u, v_{n_j})<\limsup_{j\to \infty} D_g(\bar{u}, v_{n_j})\\
&=&\lim_{n\to\infty} D_g(\bar{u}, v_n)=\limsup_{i\to\infty} D_g(\bar{u}, v_{n_i})\\
&<& \limsup_{i\to\infty} D_g(u, v_{n_i})=\lim_{n\to\infty} D_g(u, v_n),
\end{eqnarray*}
which is a contradiction. The contradiction we have reached shows that $u=\bar{u}$. Therefore $u$ is indeed unique, as claimed. Thus the whole sequence $\{v_n\}$ converges weakly to an element of $(\mathcal{A}+\mathcal{B})^{-1}(0^*)$, as asserted.
		\end{proof}
	
 We now turn to Algorithm \ref{AL1A} and establish a weak convergence theorem for it.
		\begin{theorem}\label{thm 1A}
Let  Assumption \ref{ass1}  hold and let $\alpha\in \Big(\delta,~\frac{\gamma(1-2\delta)}{2}\Big)$ for some $\delta\in (0, \frac{1}{2})$ and $~ \gamma>0.$ Then the sequence $\{v_n\}$ generated by Algorithm \ref{AL1A} converges weakly to an element of $(\mathcal{A}+\mathcal{B})^{-1}(0^*).$
		\end{theorem}
	\begin{proof}
Let $z\in (\mathcal{A}+\mathcal{B})^{-1}(0^*).$ Following a method of proof which is similar to that of the proof of Lemma \ref{lem1}, we obtain that
	\begin{align}\label{lem1A}\nonumber
	D_g(z,v_{n+1})&\leq D_g(z,v_n)+\mu_{n-1}\langle \mathcal{A}v_n-\mathcal{A}v_{n-1},z-v_n\rangle+\mu_{n-1}\langle \mathcal{A}v_n-\mathcal{A}v_{n-1},v_n-v_{n+1}\rangle\\
	&+\mu_n\langle \mathcal{A}v_n-\mathcal{A}v_{n+1}, z-v_{n+1}\rangle -D_g(v_{n+1},v_n).
	\end{align}
	From \eqref{eq1A} and \eqref{17AA} it follows that
	\begin{align}\label{a3A}\nonumber
	\mu_{n-1}\langle \mathcal{A}v_n-\mathcal{A}v_{n-1},v_n-v_{n+1}\rangle&\leq \mu_{n-1} \|\mathcal{A}v_n-\mathcal{A}v_{n-1}\|\|v_n-v_{n+1}\|\\ \nonumber
	&\leq \frac{\mu_{n-1}}{\mu_n} \alpha \|v_n-v_{n-1}\|\|v_n-v_{n+1}\|\\ \nonumber
	&\leq \frac{\mu_{n-1}}{\mu_n}\frac{\alpha}{2}\Big(\|v_n-v_{n-1}\|^2+\|v_{n+1}-v_n\|^2\Big)\\
	&\leq \frac{\mu_{n-1}}{\mu_n} \alpha \gamma^{-1}\Big(D_g(v_n,v_{n-1})+D_g(v_{n+1},v_n)\Big).
	\end{align}
				By Remark \ref{VSS}(ii) and the condition $\alpha\in \Big(\delta,~\frac{\gamma(1-2\delta)}{2}\Big),$ we have that $\lim\limits_{n\to\infty}\left(1-\frac{\mu_{n}}{\mu_{n+1}}\alpha \gamma^{-1}\right)=1-\alpha\gamma^{-1}>\frac{1}{2}+\delta$. Hence, there exists $n_0\geq 1$ such that $1-\frac{\mu_{n}}{\mu_{n+1}}\alpha \gamma^{-1}>\frac{1}{2}+\delta~\forall n\geq n_0$; which further implies that $\frac{\mu_{n}}{\mu_{n+1}}\alpha \gamma^{-1}<\frac{1}{2}-\delta~\forall n\geq n_0$. Hence, \eqref{a3A} becomes
		\begin{align}\label{a4A***}
			\mu_{n-1}	\langle \mathcal{A}v_{n}-\mathcal{A}v_{n-1},v_n-v_{n+1}\rangle 				&\leq \left(\frac{1}{2}-\delta\right)\Big(D_g(v_{n},v_{n-1})+D_g(v_{n+1},v_{n})\Big)~\forall n\geq n_0.
				\end{align}
Using \eqref{a4A***} in \eqref{lem1A}, we obtain								\begin{align}\label{a5A}\nonumber
				D_g(z,v_{n+1})&\leq D_g(z,v_n)+\mu_{n-1}\langle \mathcal{A}v_n-\mathcal{A}v_{n-1},z-v_n\rangle +\Big(\frac{1}{2}-\delta \Big)D_g(v_n,v_{n-1})\\ \nonumber
				&+\mu_n\langle \mathcal{A}v_n-\mathcal{A}v_{n+1},z-v_{n+1}\rangle -\Big(\frac{1}{2}+\delta\Big)D_g(v_{n+1},v_n)\\ \nonumber
				&\leq D_g(z,v_n)+\mu_{n-1}\langle \mathcal{A}v_n-\mathcal{A}v_{n-1},z-v_n\rangle +\frac{1}{2}D_g(v_n,v_{n-1})\\
				&+\mu_n\langle \mathcal{A}v_n-\mathcal{A}v_{n+1},z-v_{n+1}\rangle -\Big(\frac{1}{2}+\delta\Big)D_g(v_{n+1},v_n)~\forall n\geq n_0.
				\end{align}
For $n\geq n_0$, let
\begin{eqnarray*}
s_n&=&D_g(z,v_n)+\mu_{n-1}\langle \mathcal{A}v_n-\mathcal{A}v_{n-1},z-v_n\rangle +\frac{1}{2}D_g(v_n,v_{n-1}),\\
t_n&=&\delta D_g(v_{n+1},v_n).
\end{eqnarray*}
Then \eqref{a5A} can be rewritten as
\begin{eqnarray}\label{4.19*}
s_{n+1}\leq s_{n}-t_n~\forall n\geq n_0.
\end{eqnarray}
Using an argument similar to the one used regarding \eqref{4.8**}, we obtain that $s_n\geq 0 \; ~\forall n\geq n_0$. Hence, again following the arguments used concerning \eqref{4.8**} and noting that the sequence $\{\mu_n\}$ is bounded, we can show that $\{v_n\}$ indeed converges weakly to an element of $(\mathcal{A}+\mathcal{B})^{-1}(0^*),$
as asserted.
		\end{proof}

\begin{remark}
Theorem \ref{thm 1} and Theorem \ref{thm 1A} extend Theorem 2.5 and Theorem 3.4 of \cite{Malitsky2}, respectively, from Hilbert space to all reflexive Banach spaces.

\end{remark}

\hfill

\noindent If we set $\mathcal{A}= 0$ in Algorithm \ref{al1}, then we obtain the following result concerning the solution of the NPP \eqref{NPP}
as a corollary of Theorem \ref{thm 1}.

		\begin{cor}\label{COR1}
Let  Assumption \ref{ass1} hold and let $\mu>0$. For arbitrary $v_1 \in \mathbb{X}$, let the sequence $\{v_n\}$ be generated by
		 \begin{eqnarray*}
v_{n+1}={Res}^g_{\mu\mathcal{B}}(v_n),~ n\geq 1.
\end{eqnarray*}
Then $\{v_n\}$  converges weakly to an element of $\mathcal{B}^{-1}(0^*).$
		\end{cor}
\begin{remark}
We can replace $\mu$  with $\{\mu_n\}$ in Corollary \ref{COR1} to obtain a corollary of Theorem \ref{thm 1A}.
	\end{remark}
\noindent If we set $\mathcal{B}=N_\mathcal{C}$ in Algorithms \ref{al1} and \ref{AL1A} (where $N_\mathcal{C}$ is as defined in \eqref{NC}), we get $\mbox{Res}_{\mu \mathcal{B}}^g={Proj}^g_\mathcal{C}$. Hence, we obtain the following new results as corollaries of Theorem \ref{thm 1} and  Theorem \ref{thm 1A}, respectively.
These results concern the following variational inequality problem: $\mbox{Find} ~v\in \mathcal{C} ~\mbox{such that}~ \langle \mathcal{A}v,u-v\rangle \geq 0~\forall  u\in \mathcal{C},$ where $\mathcal{A}:\mathcal{C}\to \mathbb{X}^*$. \\
We denote the solution set of this problem by $VI(\mathcal{C}, \mathcal{A})$.

		\begin{cor}\label{COR3}
Let  Assumption \ref{ass1} hold and let $\mu\in \Big[\delta,~\frac{\gamma(1-2\delta)}{2L}\Big]$ for some $\delta\in (0, \frac{1}{2})$ and $~ \gamma>0.$ For arbitrary $v_0,v_1 \in \mathbb{X}$ and $\mu>0$, let the sequence $\{v_n\}$ be generated by
		 \begin{eqnarray*}
v_{n+1}={Proj}^g_\mathcal{C}\Big(\nabla g^*(\nabla g(v_n)-\mu(2\mathcal{A}v_n-\mathcal{A}v_{n-1}))\Big)~\forall n\geq 1.
\end{eqnarray*}
Then $\{v_n\}$  converges weakly to an element of $VI(\mathcal{C}, \mathcal{A})$.
		\end{cor}

		\begin{cor}\label{COR4}
Let  Assumption \ref{ass1}  hold and  let $\alpha\in \Big(\delta,~\frac{\gamma(1-2\delta)}{2}\Big)$ for some $\delta\in (0, \frac{1}{2})$ and $~ \gamma>0.$ Let $\alpha\in (0, 1)$, $\mu_0,\mu_1>0$ and choose a nonnegative real sequence $\{d_n\}$  such that $\sum_{n=1}^{\infty} d_n< \infty.$ For arbitrary $v_0,v_1 \in \mathbb{X}$, let the sequence $\{v_n\}$ be generated by
	 \begin{eqnarray*}
v_{n+1}={Proj}^g_\mathcal{C}\Big(\nabla g^*(\nabla g(v_n)-\left((\mu_n+\mu_{n-1})\mathcal{A}v_n-\mu_{n-1}\mathcal{A}v_{n-1}\right))\Big)~\forall n\geq 1,
\end{eqnarray*}
where
\begin{eqnarray}
	\mu_{n+1}=\begin{cases}
	\min \left\{\frac{\alpha\|v_n-v_{n+1}\|}{||\mathcal{A}v_n-\mathcal{A}v_{n+1}||},~\mu_{n}+d_{n}\right\}, & \mbox{if}~  \mathcal{A}v_n\neq \mathcal{A}v_{n+1},\\
	\mu_{n}+d_n,& \mbox{otherwise}.
	\end{cases}
	\end{eqnarray}
	Then  $\{v_n\}$ converges weakly to an element of $VI(\mathcal{C}, \mathcal{A}).$
		\end{cor}

\subsection{Rate of Convergence} ${}$\\
In this subsection we obtain rates of convergence for both Algorithm \ref{al1} and Algorithm \ref{AL1A}.

\noindent It follows from Algorithm \ref{al1} (or Algorithm \ref{AL1A}) that $v_{n+1}=v_n=v_{n-1}$ if and only if $v_n\in (\mathcal{A}+\mathcal{B})^{-1}(0^*)$. That is,
\begin{eqnarray*}
v_{n+1}=v_n=v_{n-1} & \Leftrightarrow& v_n={Res}^g_{\mu\mathcal{B}}\Big(\nabla g^*(\nabla g(v_n)-\mu\mathcal{A}v_n)\Big)\nonumber\\
& \Leftrightarrow & v_n=\left(\nabla g+\mu \mathcal{B}\right)^{-1} \circ \nabla g \Big(\nabla g^*(\nabla g(v_n)-\mu\mathcal{A}v_n)\Big)\nonumber\\
& \Leftrightarrow & \left(\nabla g(v_n)- \mu\mathcal{A}v_n \right)\in \left(\nabla g (v_n)+\mu \mathcal{B}v_n\right)\nonumber\\
& \Leftrightarrow & v_n\in (\mathcal{A}+\mathcal{B})^{-1}(0^*).
\end{eqnarray*}
In our theorems, we established that $||v_{n+1}-v_n||\to 0$ as $n\to\infty$ (which also means that $||v_{n}-v_{n-1}||\to 0$ as $n\to\infty$) whenever $(\mathcal{A}+\mathcal{B})^{-1}(0^*)$ is nonempty. Hence, using $||v_{n+1}-v_n||$ as a measure of the convergence rate, we obtain in the next theorem a sublinear rate of convergence for Algorithm \ref{al1}.

\begin{theorem}\label{rate}
Let  Assumption \ref{ass1} hold and let $\mu\in \Big[\delta,~\frac{\gamma(1-2\delta)}{2L}\Big]$ for some $\delta\in (0, \frac{1}{2})$ and $~ \gamma>0.$
Then $$\min\limits_{1\leq j\leq n} \|v_{j+1}-v_j\|=\mathcal{O}(1/\sqrt{n}).$$
\end{theorem}
\begin{proof} It follows from \eqref{a5} that
\begin{align*}\nonumber
\delta D_g(v_{n+1},v_n)	&\leq D_g(z,v_n)-D_g(z,v_{n+1})+\mu\langle \mathcal{A}v_n-\mathcal{A}v_{n-1},z-v_n\rangle \\
				&-\mu\langle \mathcal{A}v_{n+1}-\mathcal{A}v_n,z-v_{n+1}\rangle +\frac{1}{2}D_g(v_n,v_{n-1}) -\frac{1}{2}D_g(v_{n+1},v_n).
				\end{align*}
This implies that
\begin{align}\label{AB}\nonumber
\delta \sum_{j=1}^n D_g(v_{j+1},v_j)	&\leq D_g(z,v_1)-D_g(z,v_{n+1})+\mu\langle \mathcal{A}v_1-\mathcal{A}v_{0}, z-v_1\rangle \\ \nonumber
				&-\mu\langle \mathcal{A}v_{n+1}-\mathcal{A}v_n,z-v_{n+1}\rangle +\frac{1}{2}D_g(v_1,v_{0}) -\frac{1}{2}D_g(v_{n+1},v_n)\\
				 &= D_g(z,v_1)+\mu\langle \mathcal{A}v_1-\mathcal{A}v_{0}, z-v_1\rangle +\frac{1}{2}D_g(v_1,v_{0}) -s_{n+1},
				\end{align}
where we have $$s_{n+1}=D_g(z,v_{n+1})+\mu\langle \mathcal{A}v_{n+1}-\mathcal{A}v_{n},z-v_{n+1}\rangle +\frac{1}{2}D_g(v_{n+1},v_{n}).$$
By \eqref{4.8**}, we also have $s_{n+1}\geq 0~\forall n\geq 1.$
Hence, it follows from \eqref{AB} and \eqref{17AA} that
\begin{align*}\nonumber
\sum_{j=1}^n \|v_{j+1}-v_j\|^2	&\leq \frac{2}{\delta \gamma} \left[D_g(z,v_1)+\mu\langle \mathcal{A}v_1-\mathcal{A}v_{0}, z-v_1\rangle +\frac{1}{2}D_g(v_1,v_{0})\right].
				\end{align*}
Therefore,
$$\min\limits_{1\leq j\leq n} \|v_{j+1}-v_j\|^2	\leq \frac{2}{n\delta \gamma} \left[D_g(z,v_1)+\mu\langle \mathcal{A}v_1-\mathcal{A}v_{0}, z-v_1\rangle +\frac{1}{2}D_g(v_1,v_{0})\right].$$
This means that
$$\min\limits_{1\leq j\leq n} \|v_{j+1}-v_j\|=\mathcal{O}(1/\sqrt{n}).$$
\end{proof}
Similarly to Theorem \ref{rate}, we also have the following sublinear rate of convergence result for Algorithm \ref{AL1A}.
\begin{theorem}
Let  Assumption \ref{ass1}  hold and  let $\alpha\in \Big(\delta,~\frac{\gamma(1-2\delta)}{2}\Big)$ for some $\delta\in (0, \frac{1}{2})$ and $~ \gamma>0.$ Then $$\min\limits_{1\leq j\leq n} \|v_{j+1}-v_j\|=\mathcal{O}(1/\sqrt{n}).$$
		\end{theorem}

\subsection{Strong Convergence} ${}$\\
In order to obtain strong convergence of the sequences generated by our methods, we replace Assumption \ref{ass1} with the following conditions.

\hrule\hrule	
	\begin{ass}\label{ass1*} Let  $\mathcal{A}:\mathbb{X}\to \mathbb{X}^*$ and $\mathcal{B}:\mathbb{X}\to 2^{\mathbb{X}^*}$ be operators which satisfy the following conditions:
\hrule\hrule
\begin{itemize}
\item[(a)] $\mathcal{A}$  is monotone on $\mathbb{X}$ and $\mathcal{B}$ is maximal monotone and $\tau$-strongly monotone on $\mathbb{X}$,
\item[(a)*] $\mathcal{A}$  is $\tau$-strongly monotone on $\mathbb{X}$ and $\mathcal{B}$ is maximal monotone on $\mathbb{X}$,
\item[(b)]  $\mathcal{A}$ is Lipschitz continuous on $\mathbb{X}$ with constant $L>0$.
		\end{itemize}
\end{ass}		

\hfill

 \begin{lemma}\label{lem1S} Let a sequence $\{v_n\}$ be generated by Algorithm \ref{al1} when Assumption \ref{ass1*}(a) or (a)*  holds. Then
	\begin{align*}
\tau\mu \|v_{n+1}-z\|^2	&\leq D_g(z,v_n)-D_g(z,v_{n+1})+\mu\langle \mathcal{A}v_n-\mathcal{A}v_{n-1},z-v_n\rangle+\mu\langle \mathcal{A}v_n-\mathcal{A}v_{n-1},v_n-v_{n+1}\rangle\\
	&+\mu\langle \mathcal{A}v_n-\mathcal{A}v_{n+1}, z-v_{n+1}\rangle -D_g(v_{n+1},v_n) \hspace{0.2cm}\forall z\in (\mathcal{A}+\mathcal{B})^{-1}(0^*).
	\end{align*}
\end{lemma}
\begin{proof}
Let $z\in (\mathcal{A}+\mathcal{B})^{-1}(0^*)$. Then $-\mathcal{A}z\in \mathcal{B}z.$ Using arguments similar to those used in obtaining \eqref{**}, we get
		\begin{eqnarray}\label{**S}
		\frac{1}{\mu}\left(\nabla g(v_n)-\mu (2\mathcal{A}v_n-\mathcal{A}v_{n-1})-\nabla g(v_{n+1})\right)\in B v_{n+1}.
		\end{eqnarray}

\hfill

\noindent If we use Assumption \ref{ass1*}(a), in particular, the $\tau$-strong monotonicity of $\mathcal{B}$, we get
$$\tau \|v_{n+1}-z\|^2\leq \left\langle \frac{1}{\mu}\left(\nabla g (v_n)-\mu (2\mathcal{A}v_n-\mathcal{A}v_{n-1})-\nabla  g(v_{n+1})\right)+\mathcal{A}z, v_{n+1}-z\right\rangle,$$
	which implies that
	\begin{eqnarray*}
	\mu \tau \|v_{n+1}-z\|^2&\leq & \langle \nabla g (v_{n+1})-\nabla g (v_n), z-v_{n+1}\rangle+\mu \langle \mathcal{A}v_n-\mathcal{A}z, z-v_{n+1}\rangle+\mu \langle \mathcal{A}v_n-\mathcal{A}v_{n-1}, z-v_{n+1}\rangle.
\end{eqnarray*}		
Using the monotonicity of $\mathcal{A}$ and equation \eqref{3P}, we obtain
		\begin{align}\label{a2S}\nonumber
		\mu \tau \|v_{n+1}-z\|^2&\leq D_g(z,v_n)-D_g(z,v_{n+1})-D_g(v_{n+1}, v_n)+\mu\langle \mathcal{A}v_n-\mathcal{A}v_{n+1},z-v_{n+1}\rangle +\mu\langle \mathcal{A}v_n-\mathcal{A}v_{n-1},z-v_{n}\rangle \\
		& +\mu\langle \mathcal{A}v_n-\mathcal{A}v_{n-1},v_n-v_{n+1}\rangle,
		\end{align}
		which is the desired inequality.

\hfill		
		
\noindent On the other hand, if we use Assumption \ref{ass1*}(a)*, then by  the monotonicity of $\mathcal{B}$, we obtain
		$$0\leq \left\langle \frac{1}{\mu}\left(\nabla g (v_n)-\mu (2\mathcal{A}v_n-\mathcal{A}v_{n-1})-\nabla  g(v_{n+1})\right)+\mathcal{A}z, v_{n+1}-z\right\rangle,$$
	which implies that
	\begin{eqnarray}\label{M1S}
	0&\leq & \langle \nabla g (v_{n+1})-\nabla g (v_n), z-v_{n+1}\rangle+\mu \langle \mathcal{A}v_n-\mathcal{A}z, z-v_{n+1}\rangle+\mu \langle \mathcal{A}v_n-\mathcal{A}v_{n-1}, z-v_{n+1}\rangle.
\end{eqnarray}		
Since $\mathcal{A}$ is $\tau$-strongly monotone, it follows that
\begin{eqnarray}\label{M2S}
\langle \mathcal{A}v_n-\mathcal{A}z, z-v_{n+1}\rangle\leq \langle \mathcal{A}v_n-\mathcal{A}v_{n+1}, z-v_{n+1}\rangle-\tau\|v_{n+1}-z\|^2.
\end{eqnarray} 		
Now, using \eqref{M2S} in \eqref{M1S} and noting equation \eqref{3P}, we arrive at the desired conclusion.
\end{proof}

\begin{theorem}\label{thm 1S}
Let  Assumption \ref{ass1*}(a),(b) or (a)*,(b) hold and let $\mu\in \Big[\delta,~\frac{\gamma(1-2\delta)}{2L}\Big]$ for some $\delta\in (0, \frac{1}{2})$ and $~ \gamma>0.$
Then the sequence $\{v_n\}$ generated by Algorithm \ref{al1} converges strongly to  $u\in (\mathcal{A}+\mathcal{B})^{-1}(0^*).$
		\end{theorem}
\begin{proof}
Let $u\in (\mathcal{A}+\mathcal{B})^{-1}(0^*)$. Using similar arguments to those used in the proof of Theorem \ref{thm 1}, we obtain that $\{v_n\}$ is bounded, $\lim\limits_{n\to\infty}\|v_n-v_{n+1}\|=0,~\lim\limits_{n\to\infty}\|\mathcal{A}v_n-\mathcal{A}v_{n+1}\|=0$ and $\lim\limits_{n\to\infty}D_g(u, v_n)$ exists. Using these facts in Lemma \ref{lem1S} and replacing $z$ with $u$, we see that $\lim\limits_{n\to\infty}\|v_{n+1}-u\|\leq 0$, which implies
that $\{v_n\}$ converges strongly to  $u\in (\mathcal{A}+\mathcal{B})^{-1}(0^*),$ as asserted.
\end{proof}

\hfill

Similarly to Theorem \ref{thm 1S}, we have the following strong convergence theorem for Algorithm \ref{AL1A}.
\begin{theorem}\label{thm 1S*}
Let  Assumption \ref{ass1*}(a),(b) or (a)*,(b) hold and  let $\alpha\in \Big(\delta,~\frac{\gamma(1-2\delta)}{2}\Big)$ for some $\delta\in (0, \frac{1}{2})$ and $~ \gamma>0.$ Then the sequence $\{v_n\}$ generated by Algorithm \ref{AL1A} converges strongly to  $u\in (\mathcal{A}+\mathcal{B})^{-1}(0^*).$\\
		\end{theorem}

\begin{remark}
Similarly to Corollaries \ref{COR1}-\ref{COR4}, we can obtain strong convergence results as corollaries of Theorems \ref{thm 1S} and \ref{thm 1S*} for solving the NPP \eqref{NPP}  and the variational inequality problem in real reflexive Banach spaces.
\end{remark}

\section{Application to generalized Nash Equilbirium Problems for gas markets}\label{Se5}
\noindent In this section we apply our results to solving a Generalized Nash Equilibrium Problem (GNEP) for which the constraints are governed by a system of partial differential equations (PDEs). This type of problem is particularly useful in the trading and transporting of natural gas because the dynamics of gas transport is often posed as a system of PDEs. In \cite{GNEP}, using a generalized Nash equilibrium approach with PDE-constraints, the authors studied a two-node (production and consumption nodes) gas market.\\ Consider strategic firms that trade and transport natural gas through a pipeline system. That is, strategic gas firms that decide on their production of gas at an injection node (production node) and on their sales at a withdrawal node (consumption node) of a pipeline system with the aim of making maximum profit over a finite period of time $[0, T]$.
Here we focus on a single pipe only. The flow through the pipe which connects the injection node at $x=0$ and withdrawal node at $x=L$ (where $L>0$ is the length of the pipe) is governed by a system of PDEs with some boundary conditions. A detailed study of this system of PDEs can be found in \cite[Section 3]{GNEP}.

\noindent As  in \cite{GNEP}, let $y:=(p, q)$, where $p$ and $q$ are the pressure and mass flow, respectively, in the pipe. Note that $p$ and $q$ are the state variables in the pipe. Also, denote the collection of a firm's decisions by $u$. Then the inclusion $y=S(u)\in K$ (where the mapping $S$ provides the solution of the system of PDEs for any given $u$ and $K$ is a closed and convex set) describes the shared constraints of the firms. At any time $t\in [0, T]$, let $q^{\mbox{in}}(t)$ denote the quantity of gas injected into the pipe at the production node at a cost of $c(t)q^{\mbox{in}}(t)$. We also denote by $q^{\mbox{out}}(t)$ the quantity of gas withdrawn from the pipe at the consumption node due to the gas demand at this node. Then the price of gas is given by an inverse demand function $P(t, q^{\mbox{out}}(t)).$  We assume that firm $i$ ($i\in \{1,\dots, M\})$ strategically chooses the quantity $q_i^{\mbox{in}}(t)$ injected at the production node located at $x=0$ at a pressure $p_i^{\mbox{in}}(t)$ and the quantity $q_i^{\mbox{out}}(t)$ sold at the consumption node $x=L$ at a pressure $p_i^{\mbox{out}}(t)$ at any time $t\in [0, T]$. These decision variables are grouped into $u_i:=(p_i^{\mbox{in}},p_i^{\mbox{out}}, q_i^{\mbox{in}}, q_i^{\mbox{out}})$, which belong to a closed and convex subset $U_i$ of $L^2(0, T)^4$. \\
Thus, the maximization problem for firm $i$ is as follows::
\begin{eqnarray}
&&\max_{u_i\in U_i}\int_0^T \left[P\left(t, \sum_{k=1}^Mq_k^{\mbox{out}}(t)\right)q_i^{\mbox{out}}(t)-c_i(t)q_i^{\mbox{in}}(t)\right]dt-R_i(u_i)\label{GNEP1}\\
&& \mbox{such that} \hspace{0.3cm} 0\leq q_i^{\mbox{in}}(t)\leq \bar{q}_i^{\mbox{in}} \hspace{0.3cm} \mbox{for a.e}~~t\in (0,T), \label{GNEP2}\\
&& \int_0^T q_i^{\mbox{out}}(t)-q_i^{\mbox{in}}(t)\leq 0, \label{GNEP3}\\
&&S(u)\in K, \label{GNEP4}
\end{eqnarray}
where $\bar{q}_i^{\mbox{in}}$ is a product-specific capacity, $u=(u_i, u_{-i})$,  $u_i$ collects the decisions of firm $i$, $u_{-i}$ collects the decisions of all the other firms, and $R_i$ is a convex function. The feasible set of firm $i$ is the intersection of $U_i$ with the constraints \eqref{GNEP2}--\eqref{GNEP3} (see \cite{GNEP} for more details).

\noindent We note that the above maximization problem for strategic gas firms gives rise to the following GNEP for player $i$ ($i\in \{1, \dots, M\}$): \begin{eqnarray}\label{GNEP5}
&&\min_{u_i\in U}\int_0^T \left(\alpha(t) \sum_{k=1}^Mq_k^{\mbox{out}}(t)-\beta(t)\right)q_i^{\mbox{out}}(t)dt+g_i(u_i) \nonumber \\
&&\\
&&\mbox{such that}\hspace{0.3cm} u_i\in P^{\mbox{ad}}\times Q_i^{\mbox{ad}}~~;~~ S(u_i, u_{-i})\in K, \nonumber
\end{eqnarray}
where $g_i(u_i)=g^p(p_i)+g^q_i(q_i)$, the functions $g^p$ and $g_i^q$ are proper convex and lower semicontinuous, $\beta(t)-\alpha(t)q^{\mbox{out}}(t)=P(t, q^{\mbox{out}}(t))$, $P$ is a measurable function of $t$, the functions $\alpha$ and $\beta$ are real-valued functions of $t$, $P^{\mbox{ad}}$ and $Q_i^{\mbox{ad}}$ are subsets of $P:=H^1(0, T)^2$ (the space of pressures at the endpoints) and $Q:=H^1(0, T)^2$ (the space of mass flow at the endpoints), respectively, and $U:=P\times Q$ denotes the space of decision variables (controls) $u_i=(p_i, q_i)$ of each player.
If for each $i$, $g_i$ is coercive and the subdifferential $\partial g_i$ of $g_i$ is defined everywhere, then $P^{\mbox{ad}}$ and $Q_i^{\mbox{ad}}$ are closed and convex.\\
Clearly, the GNEP \eqref{GNEP1}--\eqref{GNEP4} that models the gas market fits into the settings of the GNEP \eqref{GNEP5}. That is, we can recover the  GNEP \eqref{GNEP1}--\eqref{GNEP4} from \eqref{GNEP5} by setting
$$g_i(u_i)=\int_0^T c_i(t)q_i^{\mbox{in}}(t)dt +R_i(u_i)$$
and $Q_i^{\mbox{ad}}$ to capture the constraints \eqref{GNEP2}--\eqref{GNEP3}.\\
The existence of solutions to the GNEP \eqref{GNEP5} for which all the pressures $p_i$ are identical was established in \cite[Theorem 4.1]{GNEP} under some conditions which are also satisfied by the GNEP \eqref{GNEP1}--\eqref{GNEP4}. \\

\noindent We now apply our algorithms in Section \ref{Se3} to solving the GNEP \eqref{GNEP5}, which in turn solves the GNEP \eqref{GNEP1}--\eqref{GNEP4} that models the gas market.  To this end, we consider some change of variables as in \cite{GNEP}.\\
 Given $y_0=(p_0, q_0)$, let $\hat{p}^0(t):=(p_0(0), p_0(L))$ and $\hat{q}^0(t):=(q_0(0)/M, q_0(L)/M)$ for any $t\in [0, T]$. Then $\hat{u}_i^{y_0}:=(\hat{p}^0, \hat{q}^0)$, and the shifted decisions for player $i$ are given by $\hat{u}_i=u_i-\hat{u}_i^{y_0}$. Let $\hat{P}:=\{\hat{p}\in P : p(0)=0\}$ and $\hat{Q}:=\{\hat{q}\in Q : q(0)=0\}$. Then $\hat{P}^{\mbox{ad}}:=\{\hat{p}\in \hat{P} : \hat{p}+\hat{p}^0\in P^{\mbox{ad}}\}$, $\hat{Q}_i^{\mbox{ad}}:=\{\hat{q}\in \hat{Q} : \hat{q}+\hat{q}^0\in Q_i^{\mbox{ad}}\}$ and $\hat{U}_i^{\mbox{ad}}=\hat{P}^{\mbox{ad}}\times \hat{Q}_i^{\mbox{ad}}$. Let $w_i=(w_i^{p0}, w_i^{pL}, w_i^{q0}, w_i^{qL})\in \hat{U}$, and define the mappings  $d_1:\hat{U} \times \hat{U}^{M-1}\to \hat{U}^*$, $d_2:\hat{U}\to \hat{U}^*$ and $e\in \hat{U}^*$ (where $\hat{U}^*$ is the dual space of $\hat{U}$) by
 $$\langle d_1(\hat{u}_i, \hat{u}_{-i}), w_i\rangle_{\hat{U}^*, \hat{U}}=\left(\alpha\sum_{k=1}^M \hat{q}_k^{\mbox{out}}, w_i^{qL} \right)_{L^2(0, T)},$$
  $$\langle d_2(\hat{u}_i), w_i\rangle_{\hat{U}^*, \hat{U}}=\left(\alpha\hat{q}_i^{\mbox{out}}, w_i^{qL} \right)_{L^2(0, T)},$$
  $$\langle e, w_i\rangle_{\hat{U}^*, \hat{U}}=\left(\beta-\hat{\alpha}q0(L), w_i^{qL} \right)_{L^2(0, T)}.$$
 Similarly, let $z_i=(z_i^{\mbox{in}}, z_i^{\mbox{out}})\in \hat{Q}$, and define the mappings  $\tilde{d}_1:\hat{Q} \times \hat{Q}^{M-1}\to \hat{Q}^*$, $\tilde{d}_2:\hat{Q}\to \hat{Q}^*$ and $\tilde{e}\in \hat{Q}^*$ (where $\hat{Q}^*$ is the dual space of $\hat{Q}$) by
 $$\langle \tilde{d}_1(\hat{q}_i, \hat{q}_{-i}), z_i\rangle_{\hat{Q}^*, \hat{Q}}=\left(\alpha\sum_{k=1}^M \hat{q}_k^{\mbox{out}}, z_i^{\mbox{out}} \right)_{L^2(0, T)},$$
  $$\langle \tilde{d}_2(\hat{q}_i), z_i\rangle_{\hat{Q}^*, \hat{Q}}=\left(\alpha\hat{q}_i^{\mbox{out}}, z_i^{\mbox{out}} \right)_{L^2(0, T)},$$
  $$\langle \tilde{e}, z_i\rangle_{\hat{Q}^*, \hat{Q}}=\left(\beta-\hat{\alpha}q0(L), z_i^{\mbox{out}} \right)_{L^2(0, T)}.$$
  Then, by the proof of Theorem 4.1 in \cite{GNEP}, we have that the inclusion
  \begin{eqnarray}\label{GNEP6}
  0^*\in \tilde{d}_1(\hat{q}_i, \hat{q}_{-i}) +\tilde{d}_2(\hat{q}_i)+\partial \hat{g}^q_i(\hat{q}_i)+N_{\hat{Q}_i^{\mbox{ad}}}(\hat{q}_i)+(\eta^{\hat{q}})
  \end{eqnarray}
 (where $\eta^{\hat{q}}\in \hat{Q}^*$ and $\hat{g}^q_i(\hat{q}_i)=g_i^q(\hat{q}_i+\hat{q}^0)$) holds for all $\hat{u}_i$, which implies that the inclusion
\begin{eqnarray}\label{GNEP7}
  &&0^*\in d_1(\hat{u}_i, \hat{u}_{-i}) +d_2(\hat{u}_i)+\partial \hat{g}_i(\hat{u}_i)+N_{\hat{U}_i^{\mbox{ad}}}(\hat{u}_i)+\hat{S}_i^*\mu-e,~\forall i\in \{1,\dots, M\}
  \end{eqnarray}
  also holds for $\hat{u}\in \hat{\mathcal{U}}:=\hat{U}^M$, where $\mu \in N_K(\hat{S}(\hat{u}))$ and $\hat{g}_i(\hat{u}_i)=g^p(\hat{p}_i+\hat{p}^0)+\hat{g}_i^q(\hat{q}_i)$. Moreover, any solution $\hat{u}$ of the inclusion \eqref{GNEP7} is a solution of the GNEP \eqref{GNEP5}.\\
  Now, let $\tilde{\mathcal{U}}:=\hat{P}\times \hat{Q}^M$ and $\tilde{Z}:=\hat{P}^*\times \prod_i\hat{Q}^*$, define the operators $\tilde{A}:\tilde{\mathcal{U}}\to \tilde{Z}$ and $\tilde{B}:\tilde{\mathcal{U}}\to 2^{\tilde{Z}}$ by
$$\tilde{A}(\tilde{p}, \hat{q}) :=  \tilde{d}_1(\hat{q}_i, \hat{q}_{-i}) +\tilde{d}_2(\hat{q}_i)~~ \mbox{and}~~\tilde{B}(\tilde{p}, \hat{q}) := \partial \hat{g}^q_i(\hat{q}_i)+N_{\hat{Q}_i^{\mbox{ad}}}(\hat{q}_i)+(\eta^{\hat{q}}).$$ Then the inclusion \eqref{GNEP6} becomes
$$0^*\in \tilde{A}(\tilde{p}, \hat{q})+\tilde{B}(\tilde{p}, \hat{q}).$$
Also,  $\tilde{A}$ is linear, monotone and demicontinuous with full domain, and $\tilde{B}$ is maximal monotone with full domain (see \cite{GNEP}). Hence, by setting $\mathcal{A}=\tilde{A}$ and $\mathcal{B}=\tilde{B}$ in Algorithm \ref{al1} or Algorithm \ref{AL1A}, we can apply our results to approximating the solutions of the GNEP.

\section{Conclusions} \label{Se6}
\noindent
In this paper we have proposed two iterative algorithms for solving the monotone inclusion problem \eqref{MVIP} in real reflexive Banach spaces. As we have seen, the methods have simple and elegant structures, and they require only one evaluation of the single-valued operator $\mathcal{A}$ at each iteration.  We have established weak convergence results when the set-valued operator $\mathcal{B}$ is maximal monotone, and $\mathcal{A}$ is monotone and Lipschitz continuous, and we obtained strong convergence results when either $\mathcal{A}$ or $\mathcal{B}$ is, in addition, required to be strongly monotone. We also obtained a rate of convergence result and applied our results to solving generalized Nash equilibrium problems in real reflexive Banach spaces.

\hfill

\noindent {\bf Acknowledgments.} \\

The second author was partially supported by the Israel Science Foundation (Grant 820/17), by the Fund for the Promotion of Research at the Technion and by the Technion General Research Fund.\\ \\

\noindent
{\bf Declarations}\\

\noindent
{\bf Conflict of interest.} The authors declare that they have no conflict of interest.\\

\noindent
{\bf Availability of data and material.} Not applicable.\\

\noindent
{\bf Code availability.} Not applicable.\\

\noindent
{\bf Authors’ contributions.} All the authors contributed to this paper.

	\end{document}